\documentclass{article}%
\usepackage{amsmath}
\usepackage{amsfonts}
\usepackage{amssymb}
\usepackage{graphicx}%
\usepackage{mathrsfs}
\DeclareMathOperator*{\esssup}{ess\,sup}
\providecommand{\U}[1]{\protect \rule{.1in}{.1in}}
\newtheorem{theorem}{Theorem}[section]

\newtheorem{definition}[theorem]{Definition}
\newtheorem{example}[theorem]{Example}

\newtheorem{lemma}[theorem]{Lemma}

\newtheorem{proposition}[theorem]{Proposition}
\newtheorem{remark}[theorem]{Remark}

\newenvironment{proof}[1][Proof]{\noindent \textbf{#1.} }{\  \rule{0.5em}{0.5em}}
\begin{document}

\title{Quadratic BSDEs with mean reflection}
\author{H\'el\`ene Hibon\thanks{IRMAR, Universit\'e Rennes 1, Campus de Beaulieu, 35042 Rennes Cedex, France.}
\and
Ying Hu\thanks{IRMAR, Universit\'e Rennes 1, Campus de Beaulieu, 35042 Rennes Cedex, France  (ying.hu@univ-rennes1.fr)
and School of
Mathematical Sciences, Fudan University, Shanghai 200433, China.}
\and Yiqing Lin\thanks{Centre de math\'{e}matiques appliqu\'{e}es,
\'{E}cole Polytechnique, 91128 Palaiseau Cedex, France.}
\and Peng Luo \thanks{Department of Mathematics, ETH Zurich, 8092 Zurich, Switzerland.
}
\and
Falei Wang\thanks{Zhongtai Securities Institute for Finance  Studies and Institute for Advanced Research, Shandong University, Jinan 250100, China (flwang@sdu.edu.cn) and IRMAR,
Universit\'e Rennes 1, Campus de Beaulieu, 35042 Rennes Cedex, France.    }}

\maketitle

\begin{abstract}
The present paper is devoted to the study of the well-posedness  of  BSDEs with mean reflection whenever the generator has quadratic growth in the $z$ argument.
This work is the sequel of \cite{BH} in which a notion of BSDEs with mean reflection is developed to tackle the super-hedging problem under running risk management constraints.
By the contraction mapping argument, we first prove that the quadratic BSDE with mean reflection admits a unique deterministic flat local solution on a small time interval whenever the terminal value is bounded. Moreover,  we build the global solution on the whole time interval by stitching local solutions when the generator is uniformly bounded with respect to the $y$ argument.
\end{abstract}


\textbf{MSC-classification}: 60H10, 60H30.
\section{Introduction}
The nonlinear Backward Stochastic Differential Equation (BSDE) of the following form was first introduced by Pardoux and Peng \cite{PP}:
\begin{align}\label{1}
 Y_t=\xi+\int^T_t f(s,Y_s,Z_s)ds-\int^T_tZ_sdB_s,\ \ \ \  \forall t\in[0,T],
\end{align}
whose solution  consists of  an adapted pair of processes $(Y,Z)$.
Pardoux and Peng have obtained  the existence and uniqueness theorem for the BSDE \eqref{1} when the generator $f$ is uniformly Lipschitz and
the terminal value $\xi$ is square integrable. Since then, researchers made great progresses in this
field. It was seen that BSDEs  have provided powerful tools for the study of mathematical finance, stochastic control and partial differential equations.
In particular, El Karoui, Peng and Quenez \cite{EKP1} have applied the BSDE theory to pricing of European contingent claims,
roughly speaking, the component $Y$ and the
component $Z$ of the solution can be interpreted as the value process of the claim and its hedging strategy, respectively.
Furthermore, El Karoui, Pardoux and Quenez  have investigated  the pricing of American claims in \cite{EKP2}. In this paper, the price of an American option can be formulated as the  ``minimal'' solution to the following  type of BSDE with constraints:
\begin{align}\label{2}
 Y_t=\xi+\int^T_t f(s,Y_s,Z_s)ds-\int^T_tZ_sdB_s+K_T-K_t,\ \ \ \  \forall t\in[0,T],
\end{align}
where $\xi$ is the terminal payoff,  the component $Y$ is forced to stay above a given running payoff $L$ and the component $K$ is adapted and non-decreasing, which describes
the cumulative consumption under the aforementioned constraint.
This constrained BSDE \eqref{2} was called reflected BSDE and has been considered by El Karoui, Kapoudjian, Pardoux, Peng and Quenez in \cite{EKP}, in which
the minimality of solution is explicitly characterized by the Skorohod type condition,
\[
\int^T_0(Y_t-L_t)dK_t=0,
\]
i.e., $K$ increases only  when $Y$ stays on the reflecting boundary $L$.

Afterwards, the theory of constrained BSDEs has been generalized in many cases in order to tackle various of pricing problems in incomplete markets,
  see, e.g., Buckdahn and Hu \cite{BH4, BH5},
 Cvitani\'{c},  Karatzas and Soner \cite{CK}, Peng and Xu \cite{PX}.
  It is also observed that the constrained BSDEs have  strong connections with the Dynkin game (cf. \cite{CK1})
  and optimal switching problems (cf. \cite{CE, HJ, HZ, HT1}).

   Generally,  the formulation of constrained BSDEs in the  aforementioned papers involves only pointwise constraints for solutions. In contrast,
    Bouchard, Elie and R\'{e}veillac \cite{BE} have introduced the so-called weak terminal condition to the BSDE framework, which says
the terminal value only satisfies a mean constraint of the form
  \[
  \mathbb{E}[\ell(Y_T-\xi)]\geq m,
  \]
where  $m$ is a given threshold, $\ell$ is a non-decreasing map and can be viewed as a loss function in quantile hedging problems or in stochastic target problems under controlled loss.

 Recently, motivated by super-hedging of claims under running risk management constraints,  Briand, Elie and Hu \cite{BH} have formulated a new type of BSDE with constraints,
 which is called the BSDE with mean reflection. In their framework, the solution $Y$  is required to satisfy the following type of mean reflection constraint:
  \[
  \mathbb{E}[\ell(t,Y_t)]\geq 0, \ \ \ \forall t\in[0,T],
  \]
 where the running loss function $(\ell(t,\cdot)_{0\leq t\leq T}$ is a collection of  (possibly random) non-decreasing real-valued map.
 This  type  of reflected equation is also closely related to  interacting particles systems,  see, e.g., Briand, Chaudru de Raynal,  Guillin and  Labart \cite{BC}.

 In order to establish the well-posedness of BSDEs with mean reflection, in \cite{BH} the authors have introduced the notion of deterministic flat solution,
 i.e., the component $K$ is a  deterministic non-decreasing process and satisfies the following type of  Skorohod condition,
\[
\int^T_0  \mathbb{E}[\ell(t,Y_t)]dK_t=0.
\]
Thanks to the restriction of non-randomness, the solution $K$ can be constructed explicitly when the generator $f$ is independent of $y$ and $z$. In this case, such a simple BSDE with mean reflection can be solved easily by applying a martingale representation type argument.
Then with the help of the fixed-point theory,
 they have generalized the result to the Lipschitz case with square integrable terminal value when the running loss function $\ell$ is bi-Lipschitz for the mean reflection.
%
Moreover, they have indicated the minimality of the deterministic flat solution
among all the deterministic solutions of (2) under an additional structural condition on the generator.

The main purpose of this paper is to study  quadratic BSDEs with mean reflection,
in which the generator has quadratic growth in $z$ and the terminal condition is bounded.
Indeed, quadratic BSDEs in the classical sense have already attracted numerous studies, which are  as powerful tools for many  finance applications,
such as  utility maximization problems  and risk sensitive control problems, see, e.g., Hu,  Imkeller  and  M\"{u}ller \cite{HI1}.

The solvability of scalar-valued quadratic BSDEs was first established by
Kobylanski \cite{K1}   via a PDE-based method under the  boundedness assumption of the terminal value.
Subsequently, Briand and Hu \cite{BH2,BH3} have extended the existence result to the case of unbounded terminal values with exponential moments
and have studied the uniqueness whenever the generator is convex (or concave). It is worth mentioning that the comparison theorem for BSDE solutions plays a key role in these works, in which the solutions are constructed by the monotone convergence.
From a different point of view, Tevzadze \cite{Te} has applied the fixed-point argument to obtain the existence and uniqueness simultaneously for quadratic BSDEs
with small terminal values and has stitched ``small'' solutions to solve a BSDE with a general bounded terminal value. In his paper,  the application of the BMO martingale theory is crucial, which was first applied in \cite{HI1} for considering quadratic BSDEs.  Apart from this,  Briand and Elie \cite{BH1}  have recently used the Malliavin calculus  to provide a probabilistic  approach for studying the quadratic BSDEs in the spirit of  \cite{A, BC}. We also refer the reader to Morlais \cite{M}, Barrieu and El Karoui \cite{BE0} for more general results beyond the Brownian framework.

Contrary to the scalar-valued case, general multi-dimensional quadratic BSDEs may not have  a solution, see Frei and dos Reis \cite{FR}. However, the result of Tevzadze \cite{Te} for small terminal values holds even for multi-dimensional cases. Besides,
%
%
Cheridito and Nam \cite{CN}  have studied  a class of multidimensional quadratic  BSDEs with special structure. Hu and Tang \cite{HT}  have discussed the local and global solutions for  multi-dimensional BSDEs with a ``diagonally'' quadratic generator.   More recently, for  multidimensional quadratic  BSDEs, Xing and Zitkovic \cite{XZ2016} have established more general existence and uniqueness
results, but in a Markovian framework, while Harter and Richou \cite{HR2016}
have obtained positive results in some general setting.

 To consider the quadratic BSDE with mean reflection, the main difficulty is the lack of a
pointwise comparison theorem for solutions (see Example \ref{6}).
In other words, it is difficult to proceed the monotone convergence argument to construct the solution as in \cite{K1, BH2, BH3}. Therefore,
we study the solvability of quadratic BSDE with mean reflection by the fixed-point argument. The key points of our method is based on the following  observation:
\begin{itemize}
\item Suppose that $(\overline{Y},\overline{Z})$ and $(Y,Z,K)$ are the solution to the BSDE \eqref{1} and the deterministic flat solution to the quadratic BSDE \eqref{2} with mean reflection, respectively. Then the uniqueness of the solution to standard BSDE \eqref{1} implies that \[
\overline{Y}_t=Y_t-K_T+K_t, \  \  \overline{Z}_t=Z_t, \ \ \ \ \ \ \forall t\in[0,T],\]
whenever the generator $f$ is independent of $y$.
\end{itemize}
Therefore, for such a simple case, we can obtain the solution in two steps: (a) solving the corresponding standard quadratic BSDE to define the component $Z$; (b) solving the BSDE with mean reflection with the generator $f(Z)$ to find the components $Y$ and $K$, where $Z$ is exactly the one obtained in the previous step.

Thanks to this preliminary result, we can define a contractive map to find the component $Y$ for solving the equation with a general quadratic generator.  Comparing with \cite{Te}, our contractive map is different such that the restriction on the size of the terminal value could be removed, however, as a first step, the constructed solution lives only locally on a small time interval. We observe that the maximal length of the time interval on which the mapping is contractive depends only on the bound of
the component $Y$. Once the component $Y$ has a uniform estimate under additional assumptions, a global solution on the whole time interval can be established by stitching the local ones.


The remainder of the paper is organized as follows. In Section 2, we recall the framework of BSDEs with mean reflection and state our main result.
Section 3 is devoted to the study the case when the generator $f$ has separable deterministic linear dependence in $y$.
The general case is investigated in Section 4,  in which we start by constructing the deterministic flat  local solutions and then stitch them to build a global
solution.

\subsubsection*{Notation.}
We introduce the notations, which will be used throughout this paper.   For  each Euclidian space, we  denote by $\langle\cdot,\cdot \rangle$  and  $|\cdot|$
 its scalar product and the associated norm, respectively.
Then consider  a  finite time horizon $[0,T]$ and  a complete probability space $(\Omega,\mathscr{F},\mathbb{P})$,  on which $B=(B_t)_{0\leq t\leq T}$ is a  standard $d$-dimensional Brownian motion.  Let $(\mathscr{F}_t)_{0\leq t\leq T}$ be the  natural filtration generated by $B$ augmented with the family $\mathscr{N}^\mathbb{P}$ of  $\mathbb{P}$-null sets of $\mathscr{F}$.
  Finally,
 we consider the following Banach spaces:
\begin{description}
\item[$\bullet$] $\mathcal{L}^{2}$ is the space of  real valued $\mathscr{F}_T$-measurable random variables $Y$
satisfying
\begin{align*}
\|Y\|_{\mathcal{L}^{2}}=\mathbb{E}[|Y|^2]^{\frac{1}{2}}<\infty;
\end{align*}
\item[$\bullet$]  $\mathcal{L}^{\infty}$  is the space of  real valued $\mathscr{F}_T$-measurable random variables $Y$
satisfying
\begin{align*}
\|Y\|_{\mathcal{L}^{\infty}}=\esssup\limits_{\omega}|Y(\omega)|<\infty;\end{align*}
\item[$\bullet$]  $\mathcal{S}^{\infty}$  is the space of  real valued  progressively measurable continuous  processes $Y$ satisfying
\begin{align*}
\|Y\|_{\mathcal{S}^{\infty}}=\esssup\limits_{(t,\omega)}|Y(t,\omega)|<\infty;
\end{align*}
\item[$\bullet$]  $\mathcal{A}^D$ is the closed subset of $\mathcal{S}^{\infty}$ consisting of deterministic non-decreasing processes $K = (K_t)_{0\leq t\leq T}$ starting from the origin;
\item[$\bullet$]
 $BMO$  is the space of  all  progressively measurable  processes $Z$ taking values in $\mathbb{R}^d$ such that
\begin{align*}
\|Z\|_{BMO} =\sup\limits_{\tau\in\mathcal{T}}\left\|\mathbb{E}_{\tau}\left[\int^T_{\tau}|Z_s|^2ds\right]\right\|_{{L}^{\infty}}^{1/2} < \infty,
\end{align*}
where $\mathcal{T}$ denotes the set of all $[0,T]$-valued stopping times $\tau$
and $\mathbb{E}_{\tau}$ is the conditional expectation with respect to $\mathscr{F}_{\tau}$.
 \end{description}

 We denote by $\mathcal{S}^{\infty}_{[a,b]}$, $\mathcal{A}^D_{[a,b]}$ and $BMO_{[a,b]}$ the corresponding spaces for the stochastic processes have time indexes on $[a,b]$.
 For each $Z\in BMO_{[a,b]}$, we set
 \[
  \mathscr{Exp}(Z\cdot B)_a^t=\exp\left(\int^t_a Z_s dB_s-\frac{1}{2}\int^t_a|Z_s|^2ds\right),
 \]
which is a martingale  by
 \cite{K}. Thus it follows from Girsanov's theorem that \newline $(B_t-\int_{a}^tZ_sds\mathbf{1}_{\{a\leq t\leq b\}})_{0\leq t\leq T}$
is a Brownian motion under the equivalent probability measure $\mathscr{Exp} (Z\cdot B)_{a}^b d\mathbb{P}$.

\section{Quadratic BSDEs with mean reflection}
In this paper, we  consider the following type of constrained BSDE:
\begin{align}\label{my1}
\begin{cases}
&Y_t=\xi+\int^T_t f(s,Y_s,Z_s)ds-\int^T_t Z_s dB_s+K_T-K_t;\\
& \ \ \ \ \ \ \ \ \ \ \ \ \ \ \ \ \ \ \mathbb{E}[\ell(t,Y_t)]\geq 0,
\end{cases} \  \  \forall t\in[0,T],
\end{align}
where the second equation is a running constraint in expectation on the component $Y$ of the solution.
The above equation is called BSDE with mean reflection, which was first introduced in \cite{BH}.
The parameters of the BSDE with mean reflection are the terminal condition $\xi$, the generator (or driver) $f$ as well as the running loss function $\ell$.
In \cite{BH}, the authors have discussed such equation under the standard Lipschitz condition on the generator and the square integrability assumption
terminal condition.

In the sequel, we  study the existence and uniqueness theorem of equation \eqref{my1} with quadratic generator and bounded terminal condition.
These parameters are supposed to satisfy the following standard running assumptions:
 \begin{itemize}
 \item[($H_\xi$)] The terminal condition $\xi$ is an  ${\mathcal{ F}}_T$-measurable random variable bounded by $L>0$ such that  \[\mathbb{E}[\ell(T,\xi)]\geq 0.\]
 \item[($H_f$)]  The driver $f:[0,T]\times\Omega\times\mathbb{R}\times\mathbb{R}^d \rightarrow\mathbb{R}$ is a $\mathcal{P}\times \mathcal{B}(\mathbb{R})\times\mathcal{B}(\mathbb{R}^d)$-measurable map such that   \begin{enumerate}
   \item[(1)] For each $t\in[0,T]$, $f(t,0,0)$ is bounded by some constant $L$, $\mathbb P$-a.s.
   \item[(2)] There exists some constant $\lambda> 0$ such that, $\mathbb P$-a.s., for all $t\in[0,T]$, for all $y$, $p\in \mathbb{R}$ and for all $z$, $q\in\mathbb{R}^d$,
	\begin{equation*}
		| f(t,y,z) - f(t,p,q)| \leq \lambda( |y-p| + (1+|z|+|q|)|z-q|),
	\end{equation*}
\end{enumerate}
where $\mathcal{P}$ denotes the sigma algebra of progressively measurable sets of $[0,T]\times\Omega$,
   $\mathcal{B}(\mathbb{R})$ and $\mathcal{B}(\mathbb{R}^d)$ are the Borel algebras on $\mathbb{R}$ and $\mathbb{R}^d$, respectively.
\item[($H_\ell$)] The running loss function $\ell : \Omega\times[0,T]\times\mathbb{R} \rightarrow \mathbb{R}$ is an $\mathscr{F}_T\times\mathcal{ B}(\mathbb{R})\times \mathcal{ B}(\mathbb{R})$-measurable map and there exists some constant $C>0$ such that, $\mathbb P$-a.s.,
\begin{enumerate}
	\item $(t,y)\rightarrow \ell(t,y)$ is continuous,
	\item $\forall t\in[0,T]$, $y\rightarrow\ell(t,y)$ is strictly increasing,
	\item  $\forall t\in[0,T]$, $\mathbb{E}[\ell(t,\infty)]>0$,
	\item $\forall t\in[0,T]$, $\forall y\in\mathbb{R} $, $|\ell(t,y)| \leq C(1+|y|)$.
\end{enumerate}
\end{itemize}

In order to introduce another assumption for the main result of this paper, we define the operator $L_t:  \mathcal{L}^2 \rightarrow [0,\infty) $, $t\in[0,T]$ by
\begin{eqnarray*}
	L_t :X \rightarrow \inf\{ x\geq 0 : \mathbb{E}[\ell(t,x+X)] \geq 0 \},
\end{eqnarray*}
 which is well-defined due to the Assumption $(H_{\ell})$, see also \cite{BH}.
  The operator $L_t$ is crucial to build a solution to  BSDEs with mean reflection.
  \begin{example}\label{4}{\upshape
  Suppose that $\ell(t,x)=x-u_t$, $t\in[0,T]$, for some given deterministic continuous process $u$. It is easy to check that
  \[
  L_t(X)=(\mathbb{E}[X]-u_t)^{-}:=-\min(\mathbb{E}[X]-u_t,0), \ \ \forall t\in[0,T]\ \text{and}\ X\in \mathcal{L}^2.
  \]
  }
  \end{example}

In addition to the aforementioned assumptions, for the construction of the solution for the quadratic BSDE with mean reflection in Section 4, the following assumptions will be needed.
 \begin{itemize}
   \item[($H_f^{\prime}$)] For each $(t,y)\in[0,T]\times\mathbb{R}$, $f(t,y,0)$ is bounded by a constant $L$, $\mathbb P$-a.s.
 \item[($H_L$)] There exists a constant $C>0$ such that for each $t\in[0,T]$,
 \begin{equation*}
|L_t(X)-L_t(Y)|\leq C\mathbb{E}[|X-Y|], \ \forall  X,Y \in \mathcal{L}^2.
\end{equation*}
\end{itemize}
\begin{remark}{\upshape
Assume that $(H_{\ell})$ holds true. Suppose that $\ell$ is a
bi-Lipschitz function in $x$, i.e., there exist some constants $0<c_{\ell}\leq C_{\ell}$ such that, $\mathbb P$-a.s., for all $t\in[0,T]$ and for all $x,y\in\mathbb{R}$,  \begin{equation*}
		{c_{\ell}} |x-y| \leq |\ell(t,x)-\ell(t,y) | \leq C_{\ell} |x-y|,
	\end{equation*}
Then $(H_L)$ holds true with $C=\frac{C_{\ell} }{c_{\ell} }$ (see also \cite{BH}).
}
\end{remark}

As in \cite{BH}, we study  deterministic flat solutions of quadratic BSDEs  with mean reflection.
\begin{definition}
A triple of processes $(Y,Z,K)\in\mathcal{S}^{\infty}\times BMO\times \mathcal{A}^D$
is said to be a deterministic  solution to the BSDE \eqref{my1} with mean reflection if it ensures that the equation \eqref{my1} holds true. A solution is said to be ``{\it flat}'' if moreover that
 $K$ increases only when needed, i.e., when we have
\begin{equation}\label{main_flat}
		\int_0^T \mathbb{E}[\ell(t,Y_t)]  dK_t = 0.
\end{equation}
\end{definition}

The first main result of this paper is on the existence and uniqueness of the local solution for the quadratic BSDE with mean reflection, which reads as follows:
\begin{theorem}\label{my16}
Assume that $(H_{\xi})-(H_f)-(H_{\ell})-(H_L)$ hold. Then, there exists a sufficiently large constant $A>0$ and a constant $0<\widehat{\delta}^A\leq T$ depending only on $A, C, L$ and $\lambda$, such that for any $h\in (0,\widehat{\delta}^A]$, the quadratic BSDE \eqref{my1} with mean reflection admits a unique
 deterministic flat solution $(Y,Z,K)\in \mathcal{S}^{\infty}_{[T-h,T]}\times BMO_{[T-h,T]}\times\mathcal{A}^D_{[T-h,T]}$ such that
\[
\|Y\|_{\mathcal{S}^{\infty}_{[T-h,T]}}  \leq A.
\]
\end{theorem}
Moreover, we stitch local solutions and obtain the solvability of the quadratic BSDE \eqref{my1} on the whole time interval under an additional condition on the generator $f$.
\begin{theorem}\label{my202}
Assume that $(H_{\xi})-(H_f)-(H_f^{\prime})-(H_{\ell})-(H_L)$ hold. Then  the quadratic BSDE \eqref{my1} with mean reflection has a
unique deterministic flat solution $(Y,Z,K)\in \mathcal{S}^{\infty}\times BMO\times\mathcal{A}^D$ on $[0, T]$. Moreover, there exists a uniform bound $\overline{L}$ depending only on $C, L,\lambda$ and $T$ such that
\[ \| {Y}\|_{\mathcal{S}^{\infty}}\leq \overline{L}.\]
\end{theorem}

\section{A simple case study}
In this section, we consider a  simple case where the generator has the following particular structure:
\begin{align}\label{myw1}
\begin{cases}
&Y_t=\xi+\int^T_t [a_sY_s+f(s,Z_s)]ds-\int^T_t Z_s dB_s+K_T-K_t;\\
& \ \ \ \ \ \ \ \ \ \ \ \ \ \ \ \ \ \ \mathbb{E}[\ell(t,Y_t)]\geq 0,
\end{cases} \ \ \forall t\in[0,T],
\end{align}
and $a$  is a deterministic and bounded measurable function.
For convenience, we rewrite the Assumption $(H_f)$ as follow:
 \begin{itemize}
\item[($H_f^*$)]  The driver $f:[0,T]\times\Omega\times\mathbb{R}^d \rightarrow\mathbb{R}$ is a $\mathcal{P}\times\mathcal{B}(\mathbb{R}^d)$-measurable map such that   \begin{enumerate}
   \item[(1)] For each $t\in[0,T]$, $f(t,0)$ is bounded by some constant $L$, $\mathbb P$-a.s.
   \item[(2)] There exists some constant $\lambda\geq 0$ such that, $\mathbb P$-a.s., for all $t\in[0,T]$ and for all $z$, $q\in\mathbb{R}^d$
	\begin{equation*}
		| f(t,z) - f(t,q)| \leq \lambda (1+|z|+|q|)|z-q|.
	\end{equation*}
\end{enumerate}
\end{itemize}

\begin{theorem}\label{myw2}
Assume that $(H_{\xi})-(H_f^*)-(H_{\ell})$ hold. Then the quadratic  BSDE \eqref{myw1} with mean reflection has a
unique deterministic flat solution $(Y,Z,K)\in \mathcal{S}^{\infty}\times BMO\times\mathcal{A}^D$.
\end{theorem}
 \begin{proof}
It suffices to prove the case where $a_s=0$ for each $t\in[0,T]$. Indeed, denote $A_t=\int^t_0a_sds$ for each $t\in[0,T]$. Then it is easy to check that
 $(Y,Z,K)$ is a deterministic flat solution to  the BSDE \eqref{myw1} with mean reflection if and only if \[
 (Y^A_t,Z^A_t,K^A_t)=\left(e^{A_t}Y_t,e^{A_t}Z_t,\int^t_0 e^{A_s}dK_s\right)
 \]is a deterministic flat solution to  the BSDE \eqref{myw1}
associated
with the parameters
 \[
 \xi^A=e^{A_T}\xi, \ \ f^A(t,z)=e^{A_t}f(t,e^{-A_t}z), \ \, \ell^A(t,y)=\ell(t,e^{-A_t}y).
 \]

We shall prove the existence and the uniqueness separately.

{\it Step 1. Existence.}
Consider the following standard quadratic BSDE  on the time interval $[0,T]$:
 \begin{align}\label{myw3}
 \overline{Y}_t=\xi+\int^T_t f(s,\overline{Z}_s)ds-\int^T_t \overline{Z}_s dB_s.
 \end{align}
By \cite{BH1} or \cite{K1}, the equation \eqref{myw3} has a unique solution
$(\overline{Y},\overline{Z})\in\mathcal{S}^{\infty}\times BMO$, which implies that
\[
\mathbb{E}\left[\left|\int^T_0|\overline{Z}_s|^2ds\right|^p\right]<\infty, \ \forall p\geq 1.
\]
Then we recall the Assumption $(H^*_f)$ and obtain that
\begin{align*}
\mathbb{E}\left[\left|\int^T_0|f(s,\overline{Z}_s)|ds\right|^2\right]\leq &\mathbb{E}\left[\left|\int^T_0(|f(s,0)|+\lambda|\overline{Z}_s|+\lambda|\overline{Z}_s|^2)ds\right|^2\right]\\
\leq & \mathbb{E}\left[\left|\int^T_0(L+\frac{1}{2}\lambda+\frac{3}{2}\lambda|\overline{Z}_s|^2)ds\right|^2\right]<\infty.
\end{align*}
Thus from Proposition 8 in \cite{BH}, the following simple BSDE with mean reflection
\begin{align*}
\begin{cases}
&\widetilde{Y}_t=\xi+\int^T_t f(s,\overline{Z}_s)ds-\int^T_t \widetilde{Z}_s dB_s+\widetilde{K}_T-\widetilde{K}_t;\\
& \ \ \ \ \ \ \ \ \ \ \ \ \ \ \ \ \ \ \mathbb{E}[\ell(t,\widetilde{Y}_t)]\geq 0,
\end{cases}\  \  \forall t\in[0,T],
\end{align*}
admits a unique deterministic flat solution $(\widetilde{Y},\widetilde{Z},\widetilde{K})$ such that
\[
\mathbb{E}\left[\sup\limits_{s\in[0,T]}|\widetilde{Y}_s|^2+\int^T_0|\widetilde{Z}_s|^2ds\right]<\infty.
\]
Moreover, for each $t\in[0,T]$  we have\begin{equation*}
	\widetilde{K}_t= \sup_{0\leq s\leq T} L_s(\overline{X}_s)- \sup_{t\leq s\leq T} L_s(\overline{X}_s)\ \text{with}\ \overline{X}_t := \mathbb{E}_t\left[\xi+\int^T_t f(s,\overline{Z}_s)ds\right].
\end{equation*}

Consequently, $(\widetilde{Y}_t-(\widetilde{K}_T-\widetilde{K}_t),\widetilde{Z}_t)_{0\leq t\leq T}$ and $(\overline{Y}_t,\overline{Z}_t)_{0\leq t\leq T}$ are both solutions to
the following standard BSDE  on the time interval $[0,T]$,
\[
\widehat{Y}_t=\xi+\int^T_t f(s,\overline{Z}_s)ds-\int^T_t\widehat{Z}_sdB_s.
\]
By the uniqueness of solutions to BSDE (see \cite{PP}), we deduce that
\[
(\overline{Y}_t,\overline{Z}_t)=(\widetilde{Y}_t-(\widetilde{K}_T-\widetilde{K}_t),\widetilde{Z}_t), \ \forall t\in[0,T].
\]
Since $\widetilde{K}$ is a deterministic continuous process, we have $\widetilde{Y}\in\mathcal{S}^{\infty}$.
Thus $(\widetilde{Y}, \widetilde{Z}, \widetilde{K})$ is a deterministic flat solution to the BSDE \eqref{myw1} with mean reflection.

 {\it Step 2. Uniqueness.}  Assume that $(Y^*,Z^*,K^*)$  is also a deterministic flat solution to the BSDE \eqref{myw1} with mean reflection.
 Then we obtain that $({Y}_t-({K}_T-{K}_t),{Z}_t)_{0\leq t\leq T}$ and $({Y}^*_t-({K}^*_T-{K}^*_t),{Z}^*_t)_{0\leq t\leq T}$ are both solutions to the standard quadratic BSDE
 \eqref{myw3}. It follows from the uniqueness of solution for the quadratic BSDE  that $Z=Z^*$ on $[0,T]$.

 Consequently, $(Y,Z,K)$ and $(Y^*,Z^*,K^*)$ are both deterministic flat solutions to the following simple BSDE with mean reflection:
 \begin{align*}
\begin{cases}
&\widetilde{Y}_t=\xi+\int^T_t f(s,Z_s)ds-\int^T_t \widetilde{Z}_s dB_s+\widetilde{K}_T-\widetilde{K}_t;\\
& \ \ \ \ \ \ \ \ \ \ \ \ \ \ \ \ \ \ \mathbb{E}[\ell(t,\widetilde{Y}_t)]\geq 0,
\end{cases}\ \ \forall t\in[0,T].
\end{align*}
Thus recalling Proposition 8 in \cite{BH} again, we derive that $(Y,Z,K)=(Y^*,Z^*,K^*)$, which ends the proof.
 \end{proof}

 We also have the minimality of
 the deterministic flat solution and the mean comparison theorem.

\begin{proposition}
Suppose that $\ell$ is strictly increasing, then a  deterministic
flat solution $(Y, Z, K)$  is minimal among all the deterministic solutions of the BSDE \eqref{myw1} with mean reflection.
\end{proposition}
\begin{proof}
By a similar argument for proving Theorem 12 in \cite{BH}, we could have the desired result.
\end{proof}
  \begin{theorem}
Suppose that $(Y^i, Z^i,K^i), i=1,2$, is the deterministic flat solution  to the BSDE \eqref{myw1} with parameters  $(\xi^i, f^i,\ell)$ satisfying the
Assumptions $(H_{\ell})-(H_L)$, where $C\leq 1$.  If one of the following conditions holds:
\begin{itemize}
\item[\emph{1}.] $\mathbb{E}[\xi^1]\geq \mathbb{E}[\xi^2]$ and $\mathbb{E}[f^1]\geq \mathbb{E}[f^2]$,
where $\xi^i$ satisfies the Assumption $(H_{\xi})$ except that the boundedness is replaced by the square integrability
and $f^i$ is independent of the component $Z$, $i=1,2$, and satisfies $(H_f^*)${\rm (1)};
\item[\emph{2}.] $\xi^i=\xi+c^i$ for some real numbers $c^1\geq c^2$ and $f^1=f^2$, where $\xi$ satisfies $(H_\xi)$ and $f^i$ satisfies $(H_f^*)$, $i=1,2$.
\end{itemize}
 then
 $\mathbb{E}[Y^1_t]\geq \mathbb{E}[Y^2_t]$, for each $t\in[0,T]$.
 \end{theorem}
 \begin{proof}
We only prove the second case, since the first can be shown in a similar fashion. We remark that in the first case, the solvability of  BSDEs with mean reflection is given by \cite{BH}.
 Without loss of generality, assume  that $a_s=0$ for each $s\in[0,T]$. Then $(Y^i_t-c^i-K_T^i+K_t^i,Z_t^i)_{0\leq t\leq T}, i=1,2$, are both solutions to
 the standard quadratic BSDE \eqref{myw3}, which implies that $Z^1=Z^2$.
 Since for each $t\in[0,T]$,
\begin{equation*}
	K^i_T-K^i_t= \sup_{t\leq s\leq T} L_s(X_s^i)\ \text{with}\ X_t^i := \mathbb{E}_t\left[\xi+c^i+\int^T_t f(s,Z_s^i)ds\right],
\end{equation*}
we deduce that
\[
K^2_T-K^2_t-(K^1_T-K^1_t)\leq \sup_{t\leq s\leq T} (L_s(X_s^2)-L_s(X_s^1))\leq C(c^1-c^2),
\]
 from which we obtain that for each $t\in[0,T]$,
 \[
 \mathbb{E}[Y^1_t-Y^2_t]=c^1-c^2+K^1_T-K^1_t-(K^2_T-K^2_t)\geq (1-C)(c^1-c^2)\geq 0.
 \]
 The proof is complete.
 \end{proof}

 Remark that in general we cannot have the pointwise comparison theorem for BSDEs with mean reflection. Indeed, let us look at the following counter-example.
 \begin{example}\label{6}{ \upshape
Consider  the following BSDE with mean reflection:\begin{align}\label{5}
\begin{cases}
&Y_t=\xi-\int^T_t Z_s dB_s+K_T-K_t,\\
& \ \ \ \ \ \ \ \  \mathbb{E}[Y_t]\geq 2T-t,
\end{cases} \  \  \forall t\in[0,T],
\end{align}
where $B$ is a $1$-dimensional Brownian motion.
Suppose that $(Y^1,Z^1,K^1)$ and $(Y^2,Z^2,K^2)$ are the deterministic flat solutions  to equation \eqref{5} corresponding to the terminal conditions $\xi^1=|B_T|^2$ and $\xi^2=\frac{3}{2}|B_T|^2$, respectively.
Then, for $t\in [0, T]$, the solutions of these two equations can be defined by
\begin{align*}
Y^1_t=|B_t|^2+2T-2t,\ \ &Z^1_t=2B_t,\ \  K^1_t=t;\\
Y^2_t=\frac{3}{2}|B_t|^2+\max{\left(2T-\frac{5}{2}t,\frac{3}{2}T-\frac{3}{2}t\right)},\ \ &Z^2_t=3B_t, \ \ K^2_t=\min{\left(t,\frac{T}{2}\right)}.
\end{align*}
Note that $\xi^1<\xi^2$ and $\mathbb{E}[Y^1_t]\leq \mathbb{E}[Y^2_t]$, $t\in[0,T]$. However, we have
\[\mathbb P(Y^1_t> Y^2_t)=\mathbb P(|B_t|^2<t)>0, \ \forall t\in\left(0,\frac{T}{2}\right].
\]
 }
 \end{example}

\section{General case}
In this section, we study the general quadratic BSDE \eqref{my1} with mean reflection. Namely, we consider this type of equation under the Assumption $(H_f)$ instead of $(H^*_f)$. As the first step, we prove the existence and uniqueness of the solution on a small time interval, which is called local solution.  Then we stitch local solutions to build the global solution.   

In order to construct a contractive map to find the local solution on a small time interval $[T-h, T]$, we assume in addition $(H_L)$ in this section. Here $h\in (0, 1)$  will be determined later. Since $L_t(0)$ is continuous in $t$,  without loss of generality we assume that $|L_t(0)|\leq L$ for each $t\in[0,T]$, see \cite{BH}.

For each $U\in  \mathcal{S}^{\infty}_{[T-h,T]}$,  it follows from Theorem \ref{myw2} that the following quadratic BSDE with mean reflection
\begin{align}\label{my2}
\begin{cases}
&Y_t^U=\xi+\int^T_t f(s,U_s,Z_s^U)ds-\int^T_t Z_s^U dB_s+K_T^U-K_t^U;\\
& \ \ \ \ \ \ \ \ \ \ \ \ \ \ \ \ \ \ \ \ \ \ \ \mathbb{E}[\ell(t,Y_t^U)]\geq 0,
\end{cases} \  \  t\in [T-h, T],
\end{align}
has a unique  deterministic flat solution  $(Y^U,Z^U,K^U)\in \mathcal{S}^{\infty}_{[T-h,T]}\times BMO_{[T-h,T]}\times\mathcal{A}^D_{[T-h,T]}$.  Then we define the purely quadratic solution map
$\Gamma: U\rightarrow\Gamma(U)$ by
\[
\Gamma(U):=Y^U, \ \forall U\in \mathcal{S}^{\infty}_{[T-h,T]}.
\]

In order to show that $\Gamma$ is contractive, for each real number $A\geq A_0$ we consider the following set:
\begin{eqnarray}
  \mathscr{B}_{A}:=\{ U\in \mathcal{S}^{\infty}_{[T-h,T]}:   \|U\|_{\mathcal{S}^{\infty}_{[T-h,T]}}\leq A \},
\end{eqnarray}
where \begin{equation}A_0=2(C+2)L+\frac{1}{2}C(\lambda+{4L}+4L\lambda)e^{6\lambda L}.\label{a}\end{equation}

\begin{lemma}\label{my3}
Assume that $(H_{\xi})-(H_f)-(H_{\ell})-(H_L)$ hold and $A\geq A_0$, where $A_0$ is defined by \eqref{a}.  Then there is a constant $\delta^A>0$ depending only on $L, \lambda, C$ and $A$ such that for any $h\in (0,\delta^A]$,
$\Gamma(\mathscr{B}_{A})\subset \mathscr{B}_{A}$.
\end{lemma}
\begin{proof}
In view of the proof of Theorem  \ref{myw2}, we conclude that for each $t\in[T-h,T]$,
\[
(Y^U_t,Z^U_t)=(\overline{Y}^U_t+(K^U_T-K^U_t),\overline{Z}^U_t),
\]
where $(\overline{Y}^U, \overline{Z}^U)\in \mathcal{S}^{\infty}_{[T-h,T]}\times BMO_{[T-h,T]}$ is the solution to the following standard BSDE on the time interval $[T-h,T]$
\begin{align}\label{my4}
\overline{Y}^U_t=\xi+\int^T_t f(s,U_s,\overline{Z}_s^U)ds-\int^T_t \overline{Z}_s^U dB_s,
\end{align}
and for each $t\in[T-h,T]$,
\begin{equation*}
	K^U_T-K^U_t= \sup_{t\leq s\leq T} L_s(X_s^U)\ \text{with}\ X_t^U := \mathbb{E}_t\left[\xi+\int^T_t f(s,U_s,Z_s^U)ds\right].
\end{equation*}
Consequently, we obtain that
 \begin{align}\label{my9}
 \|{Y}^U\|_{\mathcal{S}^{\infty}_{[T-h,T]}}\leq \|\overline{Y}^U\|_{\mathcal{S}^{\infty}_{[T-h,T]}}+ \sup_{T-h\leq s\leq T}L_s(X_s^U).
  \end{align}
The remainder of the proof will be in two steps.

{\it Step 1. The estimate of  $\overline{Y}^U$.}
Since $\overline{Z}^U\in BMO_{[T-h,T]} $, we can find a vector process $\beta\in BMO_{[T-h,T]}$  such that
\begin{eqnarray*}
f(s,U_s,\overline{Z}_s^U)-f(s,U_s,0)=\overline{Z}_s^U \beta_s, \ \forall s\in[T-h,T].
\end{eqnarray*}

Then $\widetilde {B}_t:=B_t-\int_{T-h}^t\beta_sds\mathbf{1}_{\{t\geq T-h\}}$, defines a Brownian motion under the equivalent probability measure $\widetilde{\mathbb{P}}$ given by
$$d\widetilde{\mathbb{P}}: =\mathscr{Exp} (\beta\cdot B)_{T-h}^T d\mathbb{P}.$$
Thus by the equation \eqref{my4}, we have
\[
\overline{Y}^U_t=\mathbb{E}^{\widetilde{\mathbb{P}}}_t\left[\xi+\int^T_t f(s,U_s,0)ds\right],\  \ \forall t\in [T-h, T],
\]
which implies that
\begin{align}\label{my1011}
\|\overline{Y}^U\|_{\mathcal{S}^{\infty}_{[T-h,T]}}\leq L+(L+\lambda A)h,
\end{align}
where we have used the fact that $f(s,U_s,0)$ is bounded by $L+\lambda A$. Thus recalling Proposition 2.1 in \cite{BH1},  we have
\begin{align}\label{my6}
\|Z^U\|_{BMO_{[T-h, T]}}^2\leq \frac{1}{3}\left(1+\frac{4L}{\lambda}+2\|\overline{Y}^U\|_{\mathcal{S}^{\infty}_{[T-h,T]}}\right)e^{3\lambda\|\overline{Y}^U\|_{\mathcal{S}^{\infty}_{[T-h,T]}}}.
\end{align}
{\it Step 2. The estimate of $Y^U$.}
Thanks to the Assumption $(H_L)$, for each $s\in[T-h,T]$ we have\[
|L_s(X_s^U)-L_s(0)|\leq C\mathbb{E}[|X^U_s|].
\]
 Therefore from the definition $X^U$ and Assumption ($H_f$) we deduce that
 \begin{align}
 \sup_{T-h\leq s\leq T}L_s(X_s^U)\leq &L+ C\sup_{T-h\leq s\leq T}\mathbb{E}\left[|\xi|+\int^T_s \left(|f(r,U_r, 0)|+\frac{1}{2}\lambda+\frac{3}{2}\lambda|Z_r^U|^2\right)dr\right]\nonumber\\
 \leq & (C+1)L+Ch\left(L+\lambda A+\frac{1}{2}\lambda\right)+\frac{3}{2}C\lambda\|Z^U\|_{BMO_{[T-h, T]}}^2.\label{my77}
 \end{align}
Then we define
\begin{equation}
\delta^A:=\min\left(\frac{L}{L+\lambda A+\frac{1}{2}\lambda},T\right).
\label{deltaa}
\end{equation}
Recalling equations \eqref{my9}, \eqref{my1011}, \eqref{my6} and \eqref{my77}, we derive that for each $h\in(0, \delta^A]$,
\[
 \|{Y}^U\|_{\mathcal{S}^{\infty}_{[T-h,T]}}\leq  2(C+2)L+\frac{1}{2}C(\lambda+{4L}+4L\lambda)e^{6\lambda L}=A_0\leq A,
\]
which is the desired result.
\end{proof}

Now we  show the contractive property of the purely quadratic solution map $\Gamma$.
\begin{lemma} \label{my15}
Assume that $(H_{\xi})-(H_f)-(H_{\ell})-(H_L)$ hold and $A\geq A_0$, where $A_0$ is defined by \eqref{a}. Then there exists a constant $\widehat{\delta}^A$ such that $0<\widehat{\delta}^A\leq\delta^A$ and for
 any $h\in (0,\widehat{\delta}^A]$, we have
\[\|\Gamma(U^1)-\Gamma({U}^2)\|_{\mathcal{S}^{\infty}_{[T-h,T]}}\leq \frac{1}{2}\|U^1-{U}^2\|_{\mathcal{S}^{\infty}_{[T-h,T]}}, \ \ \forall U^1, {U}^2\in \mathscr{B}_{A}.\]
\end{lemma}
\begin{proof}
For each $i=1,2$, set
\[
Y^{U^i}=\Gamma(U^i),
\]
where $(Y^{U^i}, Z^{U^i},K^{U^i})$ is the solution to the BSDE \eqref{my2} with mean reflection associated with the data $U^i$.
Applying Theorem  \ref{myw2} again, we conclude that for each $t\in[T-h,T]$,
\begin{equation}\label{repy}
(Y^{U^i}_t, Z^{U^i}_t)=(\overline{Y}^{U^i}_t+(K^{U^i}_T-K^{U^i}_t),\overline{Z}^{U^i}_t),
\end{equation}
where $(\overline{Y}^{U^i}, \overline{Z}^{U^i})\in \mathcal{S}^{\infty}_{[T-h,T]}\times BMO_{[T-h,T]} $ is the solution to the BSDE  \eqref{my4} associated with the data $U^i$.
Since for each $t\in[T-h,T]$, \begin{equation*}
	K^{U^i}_T-K^{U^i}_t= \sup_{t\leq s\leq T} L_s(X_s^{{U^i}})\ \text{with}\ X_t^{U^i} := \mathbb{E}_t\left[\xi+\int^T_t f(s,U^i_s,Z_s^{U^i})ds\right],
\end{equation*}
 we have
 \begin{align}
\sup_{T-h\leq t\leq T} |(K^{U^1}_T&-K^{U^1}_t)-(K^{U^2}_T-K^{U^2}_t)|\notag\\&\leq  \sup_{T-h\leq s\leq T}|L_s(X_s^{U^1})-L_s(X_s^{U^2})|\label{estk}\\
 &\leq C\lambda\mathbb{E}\left[\int^T_{T-h} (|U^1_r-U^2_r|+(1+|Z_r^{U^1}|+|Z_r^{U^2}|)|Z_r^{U^1}-Z_r^{U^2}|)dr\right].\notag
 \end{align}
 Applying H\"older's inequality, we obtain
 \begin{align}\label{my12}
  &\mathbb{E}\left[\int^T_{T-h}(1+|Z_r^{U^1}|+|Z_r^{U^2}|)|Z_r^{U^1}-Z_r^{U^2}|dr\right]\nonumber\\
  &\leq \sqrt{3} \mathbb{E}\left[\int^T_{T-h}(1+|Z_r^{U^1}|^2+|Z_r^{U^2}|^2)dr\right]^{\frac{1}{2}}\|Z^{U^1}-Z^{U^2}\|_{BMO_{[T-h,T]}}
\\&\leq \sqrt{3+\frac{4A}{C\lambda}}\|Z^{U^1}-Z^{U^2}\|_{BMO_{[T-h,T]}},\notag
\end{align}
 where the last inequality is deduced from the fact that $\|Z^{U^i}\|_{BMO_{[T-h,T]}}^2\leq \frac{2A}{3C\lambda}$ (see \eqref{my6}).

We recall the representation (\ref{repy}) and conclude that
  \begin{align}\label{my11}
  &\|{Y}^{U^1}-{Y}^{U^2}\|_{\mathcal{S}^{\infty}_{[T-h,T]}}\notag\\
  &\leq  \|\overline{Y}^{U^1}-\overline{Y}^{U^2}\|_{\mathcal{S}^{\infty}_{[T-h,T]}}+ \sup_{T-h\leq t\leq T} |(K^{U^1}_T-K^{U^1}_t)-(K^{U^2}_T-K^{U^2}_t)|\\
  &\leq  \|\overline{Y}^{U^1}-\overline{Y}^{U^2}\|_{\mathcal{S}^{\infty}_{[T-h,T]}}+ C\lambda (h\|{U^1}-{U^2}\|_{\mathcal{S}^{\infty}_{[T-h,T]}}+\sqrt{3+\frac{4A}{C\lambda}}\|Z^{U^1}-Z^{U^2}\|_{BMO_{[T-h,T]}}).
   \nonumber\end{align}
The remainder of the proof will be in two steps.

{\it Step 1. The estimate of $\|\overline{Y}^{U^1}-\overline{Y}^{U^2}\|_{\mathcal{S}^{\infty}_{[T-h,T]}}$.}
By the linearization argument, we can find a vector process $\widehat{\beta}\in BMO_{[T-h,T]}$ such that
\begin{align*}
f(s,U^1_s, \overline{Z}^{U^1}_s)-f(s,U^1_s, \overline{Z}^{U^2}_s)=  (\overline{Z}^{U^1}_s-\overline{Z}^{U^2}_s)\widehat{\beta}_s.
\end{align*}
Then $\widehat{B}_t:=B_t-\int_{T-h}^t\widehat{\beta}_sds\mathbf{1}_{\{t\geq T-h\}}$  defines a Brownian motion under the equivalent probability measure $\widehat{\mathbb{P}}$ given by
$$d\widehat{\mathbb{P}}: =\mathscr{Exp} (\widehat{\beta}\cdot B)_{T-h}^T d\mathbb{P}.$$
Thus by equation \eqref{my4}, we have
\[
\overline{Y}^{U^1}_t-\overline{Y}^{U^2}_t=\mathbb{E}^{\widehat{\mathbb{P}}}_t\left[\int^T_t (f(s,U^1_s,\overline{Z}^{U^2}_s)-f(s,U^2_s,\overline{Z}^{U^2}_s))ds\right], \ \forall t\in[T-h,T],
\]
which implies that
\begin{align}\label{my10}
\|\overline{Y}^{U^1}-\overline{Y}^{U^2}\|_{\mathcal{S}^{\infty}_{[T-h,T]}}\leq \lambda h\|{U^1}-{U^2}\|_{\mathcal{S}^{\infty}_{[T-h,T]}}.
\end{align}

{\it Step 2. The estimate of $\|Z^{U^1}-Z^{U^2}\|_{BMO_{[T-h,T]}}$.}
Note that for each $t\in[T-h,T]$,
\[
\overline{Y}^{U^1}_t-\overline{Y}^{U^2}_t=\int^T_t(f(s,U^1_s,\overline{Z}^{U^1}_s)-f(s,U^2_s,\overline{Z}^{U^2}_s))ds-\int^T_t(\overline{Z}^{U^1}_s-\overline{Z}^{U^2}_s)dB_s.
\]
Then applying It\^o's formula to $|\overline{Y}^{U^1}_t-\overline{Y}^{U^2}_t|^2$, we have
\[
\|Z^{U^1}-Z^{U^2}\|_{BMO_{[T-h,T]}}^2\leq 2\sup\limits_{\tau\in\mathcal{T}_{[T-h,T]}}
\left\|\mathbb{E}_{\tau}\left[\int^T_{\tau}|\overline{Y}^{U^1}_s-\overline{Y}^{U^2}_s||f(s,U^1_s,\overline{Z}^{U^1}_s)-f(s,U^2_s,\overline{Z}^{U^2}_s)|ds\right]\right\|_{\mathcal{L}^{\infty}}.
\]
By the Assumption ($H_f$), the inequalities \eqref{my12} and \eqref{my10}, we deduce that for any $h\in (0, \delta^A]$,
\begin{align*}
&\sup\limits_{\tau\in\mathcal{T}_{[T-h,T]}}
\left\|\mathbb{E}_{\tau}\left[\int^T_{\tau}|\overline{Y}^{U^1}_s-\overline{Y}^{U^2}_s||f(s,U^1_s,\overline{Z}^{U^1}_s)-f(s,U^2_s,\overline{Z}^{U^2}_s)|ds\right]\right\|_{\mathcal{L}^{\infty}}\\
&\leq \lambda^2h^2\|{U^1}-{U^2}\|_{\mathcal{S}^{\infty}_{[T-h,T]}}^2\\
& \ \ \ \ \ \ \ \ +\lambda^2 h\|{U^1}-{U^2}\|_{\mathcal{S}^{\infty}_{[T-h,T]}}\sup\limits_{\tau\in\mathcal{T}_{[T-h,T]}}\left\|
\mathbb{E}_{\tau}\left[\int^T_{\tau}(1+|Z_s^{U^1}|+|Z_s^{U^2}|)|Z_s^{U^1}-Z_s^{U^2}|ds\right]\right\|_{\mathcal{L}^{\infty}}\\
&\leq \lambda^2h^2\|{U^1}-{U^2}\|_{\mathcal{S}^{\infty}_{[T-h,T]}}^2+\sqrt{3\lambda^2+4A\lambda/C}\lambda h\|{U^1}-{U^2}\|_{\mathcal{S}^{\infty}_{[T-h,T]}}  \|Z^{U^1}-Z^{U^2}\|_{BMO_{[T-h,T]}}\\
&\leq(1+3\lambda^2+4A\lambda/C)\lambda^2 h^2\|{U^1}-{U^2}\|_{\mathcal{S}^{\infty}_{[T-h,T]}}^2+\frac{1}{4}  \|Z^{U^1}-Z^{U^2}\|_{BMO_{[T-h,T]}}^2,
\end{align*}
which together with the previous inequality implies that
\begin{equation}\label{est1}
\|Z^{U^1}-Z^{U^2}\|_{BMO_{[T-h,T]}} \leq2\sqrt{1+3\lambda^2+4A\lambda/C}\lambda  h\|{U^1}-{U^2}\|_{\mathcal{S}^{\infty}_{[T-h,T]}}.
\end{equation}
We put the estimates (\ref{my10}) and (\ref{est1}) into \eqref{my11} and obtain
\[
\|{Y}^{U^1}-{Y}^{U^2}\|_{\mathcal{S}^{\infty}_{[T-h,T]}}\leq 8\lambda h(\lambda^2+1)(A+C+1) \|{U^1}-{U^2}\|_{\mathcal{S}^{\infty}_{[T-h,T]}},
\]
where we note
\[
2\sqrt{1+3\lambda^2+4A\lambda/C}\sqrt{3\lambda^2+4A\lambda/C}\leq 1+6\lambda^2+8A\lambda/C.
\]

Now we define
\begin{equation}\label{defdelta}
\widehat{\delta}^A:=\min\left( \frac{1}{ 16\lambda(\lambda^2+1)(A+C+1) }, \delta^A\right),
\end{equation}
and it is straightforward to check that for any $h \in (0, \widehat{\delta}^A]$,
\[
\|{Y}^{U^1}-{Y}^{U^2}\|_{\mathcal{S}^{\infty}_{[T-h,T]}}\leq \frac{1}{2} \|{U^1}-{U^2}\|_{\mathcal{S}^{\infty}_{[T-h,T]}},
\]
which completes the proof.
\end{proof}

\medskip
Now we are in a position to prove Theorem \ref{my16}.

\begin{proof}[Proof of Theorem \ref{my16}]
We take $A\geq A_0$ and choose $\widehat{\delta}^A$ as (\ref{defdelta}). For $h\in (0,\widehat{\delta}^A]$, define $Y^0=0$ and by (\ref{my4}), define $(Y^1, Z^1, K^1):=(Y^{Y^0}, Z^{Y^0}, K^{Y^0})$. By recurrence, for each $i\in \mathbb{N}$, set
\begin{equation}\label{yconv}
(Y^{i+1}, Z^{i+1}, K^{i+1}):=(Y^{Y^i}, Z^{Y^i}, K^{Y^i}).
\end{equation}
It follows from Lemma \ref{my15} that there exists $Y\in \mathscr{B}_A$ such that $$\|{Y}^{n}-{Y}\|_{\mathcal{S}^{\infty}_{[T-h,T]}}\longrightarrow 0.$$
By (\ref{est1}) and (\ref{estk}), there exist $Z\in BMO_{[T-h,T]}$ and $K\in \mathcal{A}^D_{[T-h, T]}$ such that
\begin{equation}\label{zconv}
\|{Z}^{n}-{Z}\|_{BMO_{[T-h,T]}}\longrightarrow 0 \qquad {\rm and}\qquad \|{K}^{n}-{K}\|_{\mathcal{S}^{\infty}_{[T-h,T]}}\longrightarrow 0.
\end{equation}
By a standard argument, we have for each $t\in [T-h, T]$, in $\mathcal{L}^2$,
$$
\int^T_t f(s, Y^n_s, Z^{n+1}_s)ds\longrightarrow \int^T_t f(s, Y_s, Z_s)ds.
$$
Thus, the triple $(Y, Z, K)$ is a solution to the BSDE (\ref{my1}) with mean reflection and we only need to prove that the solution $(Y, Z, K)$ is ``flat''. Indeed, it is easy to check that
$$
K_T-K_t =\lim_{n\rightarrow \infty} K^n_T-K^n_t
=\lim_{n\rightarrow \infty}\sup_{t\leq s\leq T} L_s(X^n_s),$$
where $X_t^n := \mathbb{E}_t\left[\xi+\int^T_t f(s,Y^{n-1}_s,Z^n_s)ds\right].$ By (\ref{yconv}) and (\ref{zconv}), we deduce
$$\|X^n-X\|_{{\mathcal{S}^{\infty}_{[T-h,T]}}}\longrightarrow 0,$$ where
 $X_t := \mathbb{E}_t\left[\xi+\int^T_t f(s,Y_s,Z_s)ds\right].$ Therefore, $K_T-K_t =
\sup_{t\leq s\leq T} L_s(X_s)$, which implies the ``flatness''. Similarly to the Step 2 of the proof for Theorem  \ref{myw2}, we deduce the uniqueness by recalling Theorem 9 in \cite{BH}. The proof is complete.
\end{proof}
\medskip

In what follows, we construct a global solution to the quadratic BSDE \eqref{my1} with mean reflection on the whole interval $[0,T]$ by backward recursion in time, namely, we prove Theorem \ref{my202}.
To this end, we  first observe that any local solution $(Y,Z,K)$ on $[T-h,T]$ of the BSDE \eqref{my1}  with mean reflection has a uniform estimate if we assume additionally $(H_f^{\prime})$.

\begin{lemma}\label{my20}
Assume that $(H_{\xi})-(H_f)-(H_f^{\prime})-(H_{\ell})-(H_L)$ hold and the BSDE \eqref{my1} with mean reflection has a local solution $(Y,Z,K)\in \mathcal{S}^{\infty}_{[T-h,T]}\times BMO_{[T-h,T]}\times\mathcal{A}^D_{[T-h,T]}$ on $[T-h, T]$ for some $0<h\leq T$.
Then there exists a constant $\overline{L}$ depending only on $C, L, \lambda$ and $T$ such that
\[
\|Y\|_{\mathcal{S}^{\infty}_{[T-h,T]}}\leq \overline{L}.
\]
\end{lemma}
\begin{proof}
Note that
\[
(Y_t,Z_t)=(\overline{Y}_t+(K_T-K_t),\overline{Z}_t),  \ \forall t\in [T-h,T].
\]
Then the couple $(\overline{Y},\overline{Z})\in\mathcal{S}^{\infty}\times BMO$ is the solution to
the following standard quadratic BSDE on $[T-h,T]$:
\[
\overline{Y}_t=\xi+\int^T_t f(s,Y_s,\overline{Z}_s)ds-\int^T_t \overline{Z}_s dB_s.
\]
Similarly to (\ref{my9}), we obtain that for each $t\in [T-h, T]$,
 \begin{align*}
\|Y_t\|_{\mathcal{L}^{\infty}} \leq \|\overline{Y}_t\|_{\mathcal{L}^{\infty}} + \sup_{t\leq s\leq T}L_s(X_s),
  \end{align*}
where
$$
 X_t := \mathbb{E}_t\left[\xi+\int^T_t f(s,Y_s,Z_s)ds\right].
$$
By the Assumption $(H_f^{\prime})$, we have 
 \begin{align}\label{my63}
 \|Y_t\|_{\mathcal{L}^{\infty}}& \leq \|\overline{Y}_t\|_{\mathcal{L}^{\infty}}
 +(C+1)L+ C\mathbb{E}\left[\int^T_t \left(|f(r,Y_r, 0)|+\frac{1}{2}\lambda+\frac{3}{2}|Z_r|^2\right)dr\right]\notag\\
 &\leq   \|\overline{Y}_t\|_{\mathcal{L}^{\infty}} +(C+1)L+C\left(L+\frac{1}{2}\lambda\right)h+\frac{3}{2}C\|Z\|_{BMO_{[T-h, T]}}^2.
 \end{align}
We  recall Proposition 2.2 in \cite{BH1} to have
\begin{align}\label{my62}
|\overline{Y}_t|\leq L(T+1)e^{\lambda T}:=\overline{L}^1, \ \forall t\in[T-h,T].
\end{align}
Recalling again Proposition 2.1 in \cite{BH1}, we have
\begin{align}\label{my61}
\|\overline{Z}\|_{BMO_{[T-h, T]}}^2\leq \frac{1\vee T}{3}\left(1+\frac{4L}{\lambda}+2\overline{L}^1\right)e^{3\lambda \overline{L}^1}:=\overline{L}^2.
\end{align}
We derive from the inequalities \eqref{my63}, \eqref{my62} and \eqref{my61}  that for each $t\in [T-h, T]$,
\[
\|Y_t\|_{\mathcal{L}^{\infty}}\leq \overline{L}^1+(C+1)L+C\left(L+\frac{1}{2}\lambda\right)T+\frac{3}{2}C\overline{L}^2:=\overline{L},
\]
which is the desired result.
\end{proof}


\medskip
Now we are ready to prove Theorem \ref{my202}.

\begin{proof}[Proof of Theorem \ref{my202}] We treat with the existence and the uniqueness
 separately.

{\it Step 1. Existence.}
Define $\overline{A}_0=2(C+2)\overline{L}+\frac{1}{2}C(\lambda+{4\overline{L}}+4\overline{L}\lambda)e^{6\lambda \overline{L}}$. Then by Theorem \ref{my16},
there exists some constant $\overline{h}>0$ depending only on $\overline{L}, \lambda$ and $C$ together with $\overline{A}_0$ such that the quadratic BSDE \eqref{my1} with mean reflection admits a unique deterministic flat  solution
$(Y^1,Z^1,K^1)\in \mathcal{S}^{\infty}_{[T-\overline{h},T]}\times BMO_{[T-\overline{h},T]}\times\mathcal{A}^D_{[T-\overline{h},T]}$ on the  time interval $[T-\overline{h},T]$.  Furthermore, it follows from Lemma \ref{my20} that $\|Y^1\|_{\mathcal{S}^{\infty}_{[T-\overline{h},T]}}\leq \overline{L}$.

Next we take  $T -\overline{ h}$ as the terminal time and
apply Theorem \ref{my16} again to find the unique deterministic flat  solution of the BSDE  \eqref{my1} with mean reflection
$(Y^2,Z^2,K^2)\in \mathcal{S}^{\infty}_{[T-2\overline{h},T-\overline{h}]}\times BMO_{[T-2\overline{h},T-\overline{h}]}\times\mathcal{A}^D_{[T-2\overline{h},T-\overline{h}]}$ on the
time interval $[T-2\overline{h},T-\overline{h}]$.
Let us set
\[
{Y}_t=\sum\limits_{i=1}^2Y^i_t\mathbf{1}_{[T-i\overline{h},T-(i-1)\overline{h})}+Y^1_T\mathbf{1}_{\{T\}}, \  {Z}_t=\sum\limits_{i=1}^2Z^i_t\mathbf{1}_{[T-i\overline{h},T-(i-1)\overline{h})}+
Z^1_T\mathbf{1}_{\{T\}}
\]
on  $[T-2\overline{h},T]$ and  ${K}_t=K^2_t$ on $[T-2\overline{h},T-\overline{h})$, ${K}_t=K^2_{T-\overline{h}}+K^1_t$ on $[T-\overline{h},T]$. One can easily check that
$( {Y}, {Z}, {K})\in\mathcal{S}^{\infty}_{[T-2\overline{h},T]}\times BMO_{[T-2\overline{h},T]}\times\mathcal{A}^D_{[T-2\overline{h},T]}$ is a deterministic flat
solution to BSDE \eqref{my1} with mean reflection. By Lemma \ref{my20} again, it yields
 that $\| {Y}\|_{\mathcal{S}^{\infty}_{[T-2\overline{h},T]}}\leq \overline{L}$.

Furthermore, we repeat this procedure so that we can build a deterministic flat
solution  $(Y,Z,K)\in\mathcal{S}^{\infty}\times BMO\times\mathcal{A}^D $ to the quadratic BSDE \eqref{my1} with mean reflection on $[0,T]$. Moreover, it follows from Lemma \ref{my20}
that  $\| {Y}\|_{\mathcal{S}^{\infty}}\leq \overline{L}$.

{\it Step 2. Uniqueness.}  The uniqueness of the global solution on $[0, T]$ is inherited from the uniqueness of local solution on each time interval. Indeed, for each global solution $(Y, Z, K)$
to the quadratic BSDE \eqref{my1} with mean reflection,
it is easy to check that $({Y}\mathbf{1}_{[T-\widehat{h}, T-\widetilde{h}]}, {Z}\mathbf{1}_{[T-\widehat{h}, T-\widetilde{h}]},({K}_\cdot-{K}_{T-\widehat{h}})\mathbf{1}_{[T-\widehat{h}, T-\widetilde{h}]})$
defines a deterministic flat
solution to the BSDE \eqref{my1} with mean reflection associated with the terminal value $Y_{T-\widetilde{h}}$ on  the time interval $[T-\widehat{h}, T-\widetilde{h}]$, where $0\leq \widetilde{h}\leq \widehat{h}\leq T$.
The proof is complete.
\end{proof}

\begin{remark}{\upshape
Since the component $K$ is a deterministic process, by a truncation argument and the
 Malliavin calculus technique, we can find a uniform bound for $Z$ when the corresponding Malliavin derivatives are bounded, similar to Cheridito and Nam \cite{CN1}. We state the result in the Appendix and leave the proof to interested readers. We remark that under the boundedness assumption of the corresponding Malliavin derivatives, the boundedness of the terminal condition is not necessary.
 }
\end{remark}

 \noindent \textbf{Acknowledgement}:  Y. Hu's research is partially  supported
by Lebesgue Center of Mathematics ``Investissements d'avenir"
Program (No. ANR-11-LABX-0020-01), by ANR CAESARS (No. ANR-15-CE05-0024) and by ANR MFG (No. ANR-16-CE40-0015-01). Y. Lin's research is
partially  supported by  the European Research Council under grant 321111.
 P. Luo's research is partially supported by National Science Foundation of China ``Research Fund for International Young Scientists''(No. 11550110184) and by National Natural Science
Foundation of China (No. 11671257).
F. Wang's research is partially  supported by the National Natural Science Foundation of China  (No.11601282 and 11526205), by the Shandong Provincial Natural Science Foundation (No. ZR2016AQ10) and by the China
Scholarship Council (No. 201606225002).
\appendix
\renewcommand\thesection{Appendix}
\section{ }
\renewcommand\thesection{A}

Let us recall usual notations about Malliavin calculus, which can be found in  \cite{N} and \cite{CN1}.
We denote by $D_t \xi$, $0\leq t\leq T$ the  Malliavin derivative of a Malliavin differentiable random variable $\xi$, by $\mathcal{D}^{1,2}$ the completion of the class of $\mathbb{R}$-valued smooth random variables $\xi$  with respect to the norm \[
\|\xi\|_{\mathcal{D}^{1,2}} = \left|\mathbb{E}\left[\xi^2 + \int_0^T |D_t \xi|^2 dt\right]\right|^{1/2}, \]
by  $\mathcal{H}^{p}$, $p\geq 1$  the space of all  progressively measurable  processes $Y$ taking values in $\mathbb{R}$ such that
\begin{align*}
\|Y\|_{\mathcal{H}^{p}}=\left|\mathbb{E}\left[\left|\int^T_0|Y_s|^2ds\right|^{\frac{p}{2}}\right]\right|^{\frac{1}{p}}<\infty,
\end{align*}
and by $\mathcal{L}^{1,2}_a$ the space of all processes $Y\in\mathcal{H}^2$ such that $Y_t\in\mathcal{D}^{1,2}$ for each $t\in[0,T]$,
 the
process $DY_t$ admits a square integrable progressively measurable version and
\[
\|Y\|_{\mathcal{L}^{1,2}_a}=\|Y\|_{\mathcal{H}^{2}}+\mathbb{E}\left[\int^T_0\int^T_0 |D_rY_t|^2drdt\right]<\infty.
\]
Then we consider the following assumptions on the parameters:
\begin{itemize}
 \item[($\widetilde{H}_\xi$)] The terminal condition $\xi\in \mathcal{D}^{1,2}$ satisfies $|D_t\xi|\leq L$, $dt\otimes d\mathbb{P}$-a.e. for all $t\in[0,T]$ and
  $\mathbb{E}[\ell(T,\xi)]\geq 0$,
 \item[($\widetilde{H}_f$)]  For each pair $(y,z) \in \mathbb{R} \times \mathbb{R}^d$ with
$$
|z| \le Q := \sqrt{d} \left(Le^{\lambda T} + \int_0^T q_t e^{\lambda t} dt\right),
$$
it holds that
\begin{enumerate}
 \item $f(\cdot,y,z) \in \mathcal{L}^{1,2}_a$ and $|D_r f(t,y,z)| \leq q_t$, $dr \otimes d\mathbb{P} \text{-a.e.}$ for all $t\in[0,T]$ and some Borel-measurable function $q: [0,T]\rightarrow [0, \infty)$   satisfying $\int_{0}^{T} |q_t|^2 dt <\infty$,
 \item for every $r \in[0,T]$, and for all $y, p\in \mathbb{R}$,
$$
|D_r f(t,y,z) - D_r f(t,p,q)| \le K_{t}^r(|y-p|+|z-q|), \ \forall |z|,|q|\leq Q,
$$
for all $t \in [0,T]$ and some non-negative process $K_{\cdot}^r$ in $\mathcal{H}^4$.
\end{enumerate}
\end{itemize}
\begin{theorem}
Assume $(\widetilde{H}_\xi)-(H_f)-(\widetilde{H}_f)-(H_{\ell})-(H_L)$ hold.
Then  quadratic BSDE \eqref{my1} with mean reflection has a unique deterministic flat solution $(Y,Z,K)$ such that $Y$ is a continuous adapted process satisfying $\mathbb{E}[\sup\limits_{s\in[0,T]}|Y_s|^2]<\infty$, $Z$ is a bounded progressively measurable process
and $K\in\mathcal{A}^D$. Moreover,  it holds that
\[
|Z_t| \leq \sqrt{d} \left(Le^{\lambda(T-t)}+ \int_t^T q_s e^{-\lambda(t-s)} ds \right) , \  dt \otimes d\mathbb{P}\text{-a.e.}
\]
\end{theorem}

\end{document}